\documentclass[10.5pt,oneside,reqno]{amsart}
\usepackage{textcomp}
\usepackage{amsmath}
\usepackage{mathrsfs}
\usepackage{amsthm}
\usepackage[linktocpage]{hyperref}
\usepackage{amsfonts,graphics,amsthm,amsfonts,amscd,latexsym}
\usepackage{epsfig}
\usepackage{tikz}
\usepackage{tikz-cd}
\usepackage{wasysym}
\usepackage{flafter}
\usepackage{mathtools}
\usepackage{comment}
\usepackage{stmaryrd}
\usepackage{enumitem}
\usepackage{epigraph}

\hypersetup{
colorlinks=true,
linkcolor=red,
citecolor=green,
filecolor=blue,
urlcolor=blue
}
\usepackage{tikz}
\usetikzlibrary{graphs,positioning,arrows,shapes.misc,decorations.pathmorphing}

\tikzset{
>=stealth,
every picture/.style={thick},
graphs/every graph/.style={empty nodes},
}

\tikzstyle{vertex}=[
draw,
circle,
fill=black,
inner sep=1pt,
minimum width=5pt,
]
\usepackage[position=top]{subfig}
\usepackage[alphabetic]{amsrefs}
\usepackage{amssymb}
\usepackage{color}

\calclayout

\usetikzlibrary{decorations.pathmorphing}
\tikzstyle{printersafe}=[decoration={snake,amplitude=0pt}]

\newcommand{\mult}{\operatorname{mult}}

\newcommand{\ord}{\operatorname{ord}}

\newcommand{\vol}{\operatorname{vol}}

\newcommand{\oo}{\mathcal{O}}

\newcommand{\pp}{\mathbb{P}}
\newcommand{\aaa}{\mathbb{A}}

\def\O#1.{\mathcal {O}_{#1}}
\def\pr #1.{\mathbb P^{#1}}
\def\af #1.{\mathbb A^{#1}}
\def\ses#1.#2.#3.{0\to #1\to #2\to #3 \to 0}
\def\xrar#1.{\xrightarrow{#1}}
\def\K#1.{K_{#1}}
\def\bA#1.{\mathbf{A}_{#1}}
\def\bM#1.{\mathbf{M}_{#1}}
\def\bL#1.{\mathbf{L}_{#1}}
\def\bB#1.{\mathbf{B}_{#1}}
\def\bK#1.{\mathbf{K}_{#1}}
\def\subs#1.{_{#1}}
\def\sups#1.{^{#1}}

\usepackage{tikz}
\usetikzlibrary{matrix,arrows,decorations.pathmorphing}

\newtheorem{theorem}{Theorem}[section]
\newtheorem{lemma}[theorem]{Lemma}
\newtheorem{proposition}[theorem]{Proposition}
\newtheorem{corollary}[theorem]{Corollary}

\theoremstyle{definition}
\newtheorem{definition}[theorem]{Definition}

\newtheorem{remark}[theorem]{Remark}

\theoremstyle{remark}

\numberwithin{equation}{section}

\usepackage[all]{xy}

\newcounter{rownumber}[figure]
\newcounter{rownumber-irr}[figure]
\newcounter{rownumber-p1}[figure]
\begin{document}

\title{K-stability of Fano threefolds of rank $3$ and degree $14$}
\begin{abstract}
We prove that all general smooth Fano threefolds of Picard rank $3$ and degree $14$ are K-stable, where the generality condition is stated explicitly.
\end{abstract}

\author[G.~Belousov]{Grigory Belousov}
\address{Bauman Moscow State Technical University, Moscow, Russia}
\email{belousov\_grigory@mail.ru}

\author[K.~Loginov]{Konstantin Loginov}
\address{Steklov Mathematical Institute of Russian Academy of Sciences, Moscow, Russia.
National Research University Higher School of Economics, Russian Federation, Laboratory of Algebraic Geometry, NRU HSE.
Centre of Pure Mathematics, MIPT, Moscow, Russia}
\email{loginov@mi-ras.ru}

\maketitle

\section{Introduction}

We work over the field of complex numbers. Three-dimensional smooth Fano varieties are known to belong to $105$ deformation families. In \cite{ChAll}, the problem of characterising K-stable Fano threefolds was solved for a general element in each of these families. In particular, it was proven that a general smooth Fano threefolds with Picard rank $3$ and of degree $14$ is K-stable, see \cite[5.11]{ChAll}. To this end, the authors showed K-stability of one particular Fano threefold in this family, and then used the fact that K-stabiilty is an open condition in families. In this paper, we show that all general smooth Fano threefolds of Picard rank $3$ and degree $14$ are K-stable, where the generality condition is stated explicitly.

Recall that a smooth Fano threefold $X$ with Picard rank $3$ and of degree $14$ can be realized as a divisor in the linear system $|L^{\otimes2}\otimes p^*\oo(2,3)|$ in the $\pp^2$-bundle
\[
Q=\mathbb{P}_{\mathbb{P}^1\times\mathbb{P}^1}(\oo\oplus\oo(-1,-1)^{\oplus2})
\]
over $\mathbb{P}^1\times\mathbb{P}^1$ where $L$ is the tautological bundle and $p\colon Q\to \mathbb{P}^1\times\mathbb{P}^1$ is the natural projection. We denote the natural conic bundle structure on $X$ by $\pi=p|_X\colon X\to \mathbb{P}^1\times\mathbb{P}^1$.
Also, we have two fibrations $\pi_1\colon X\rightarrow\pp^1$ and $\pi_2\colon X\rightarrow\pp^1$ into del Pezzo surfaces. Let $F_1$ and $F_2$ be general fibers of $\pi_1$ and $\pi_2$, respectively. Then $F_1$ is a del Pezzo surface of degree six and $F_2$ is a del Pezzo surface of degree three.
For more details see Section \ref{sec-prelim}.

Note that $X$ is a trigonal Fano variety, which means that the anti-canonical divisor $-K_X$ is very ample, but the image of the map given by the linear system $|-K_X|$ is not the intersection of quadrics. This can be seen, for example, by restricting $-K_X$ to a smooth cubic surface which is a general fiber of $\pi_2$. However, this family of Fano varieties was erroneously omitted in Iskovskikh’s list of trigonal Fano threefolds.

For a smooth Fano variety $X$ of Picard rank $3$ and degree $14$, we formulate the following generality condition:

\

$(\star)$\quad For any multiple fiber $C'=2C$ of $\pi$, the fiber $F_2$ of $\pi_2$ that contains $C$ has only $A_1$ singularities along $C$.

\

The meaning of this condition is that, if $\pi$ has multiple fibers, then the singularities of $\pi_2$ are general along it.
Our main result is as follows:

\begin{theorem}
\label{main-theorem}
Let $X$ be a smooth Fano threefold with Picard rank $3$ and degree $14$ such that the condition $(\star)$ is satisfied. Then $X$ is K-stable.
\end{theorem}

Actually, we expect that any Fano threefold $X$ with Picard rank $3$ and of degree $14$ is K-stable. However, we are unable to prove this at the moment. Since multiple fibers of $\pi$ correspond to singular points of the discriminant curve $\Delta\subset \mathbb{P}^1\times\mathbb{P}^1$ of $\pi$, we have the following

\begin{corollary}
If the discriminant curve of $\pi\colon X\to \mathbb{P}^1\times\mathbb{P}^1$ is smooth, then $X$ is K-stable.
\end{corollary}

Another corollary of Theorem \ref{main-theorem} is as follows.

\begin{corollary}
\label{cor-a1-a2}
If singular fibers of $\pi_1$ and $\pi_2$ have singular points of type $A_1$ and $A_2$, then $X$ is K-stable.
\end{corollary}

In fact, one can prove that a general variety $X$ such that $\pi$ has a multiple fiber, satisfies the condition $(\star)$, and so by Theorem \ref{main-theorem} it is K-stable.

\

\textbf{Acknowledgements.}
The authors thank Ivan Cheltsov for proposing the problem and for useful discussions, and Constantin Shramov for helpful comments on the draft of the paper. The work of the second author was performed at the Steklov International Mathematical Center and supported by the Ministry of Science and Higher Education of the Russian Federation (agreement no. 075-15-2022-265), by
the HSE University Basic Research Program, and the Simons Foundation. The work is supported by the state assignment of MIPT (project FSMG-2023-0013). The author is a Young Russian Mathematics award winner and would like to thank its sponsors and jury. The authors thank the anonymous referee for many helpful remarks on the first version of the paper.

\section{Del Pezzo surfaces}

In this section, we collect some elementary facts on the geometry of del Pezzo surfaces.

\begin{remark}
\label{lem-duVal_contractions}
Let $S$ be a normal Gorenstein del Pezzo surface with at worst du Val singularities. Then every birational contraction of relative Picard rank $1$ is a contraction of a $(-1)$-curve $C$ where by a $(-1)$-curve we mean that $C$ is a smooth rational curve with $K_S\cdot C=-1$.
\end{remark}
\subsection{Sextic del Pezzo surface with du Val singularities}
\label{subsect-sextic}
Let $Y$ be a sextic del Pezzo surface with du Val singularities. Denote by $Y'\to Y$ the minimal resolution of $Y$. According to \cite{HW81}, there exists a morphism $Y'\to\mathbb{P}^2$ which is a composition of blow ups of smooth points. We have the following possibilities, cf. \cite[Proposition 8.3]{CoTs88} (in the diagrams below, $\bullet$ denotes a $(-1)$-curve, $\circ$ denotes a $(-2)$-curve, an edge represents that the corresponding curves intersect):\begin{enumerate}
\item
$Y$ is smooth, in which case $Y'=Y$, $Y'$ is a blow up of $3$ non-collinear points on $\mathbb{P}^2$, $\rho(Y)=4$, and $Y'$ has $6$ $(-1)$-curves.
$$
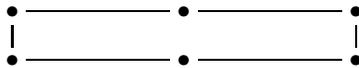
\xymatrix@R=0.8em{
\\
\bullet\ar@{-}[rr]\ar@{-}[d]&&\bullet\ar@{-}[rr]&&\bullet\ar@{-}[d]\\
\bullet\ar@{-}[rr]&&\bullet\ar@{-}[rr]&&\bullet\\
}
$$

\item
$Y$ has a unique singular point of type ${A}_1$, $Y'$ is a blow up of a point and two infinitely near points on $\mathbb{P}^2$ such that these $3$ points are not collinear. In this case $\rho(Y)=3$, and $Y'$ has $4$ $(-1)$-curves.
$$
\xymatrix@R=0.8em{
\\
\bullet\ar@{-}[r]&\bullet\ar@{-}[r]&\circ\ar@{-}[r]&\bullet\ar@{-}[r]&\bullet\\
}
$$

\item
$Y$ has a unique singular point of type ${A}_1$, $Y'$ is a blow up of $3$ collinear points on $\mathbb{P}^2$. In this case $\rho(Y)=3$, and $Y'$ has $3$ $(-1)$-curves.
$$
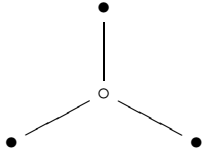
\xymatrix@R=0.8em{
&\bullet\ar@{-}[dd]&\\
&&\\
&\circ\ar@{-}[rd]\ar@{-}[ld]&\\
\bullet&&\bullet\\
}$$

\item
$Y$ has one singular point of type ${A}_2$, $Y'$ is a blow up of $3$ infinitely near points, $\rho(Y)=2$, and $Y'$ has $2$ $(-1)$-curves.
$$
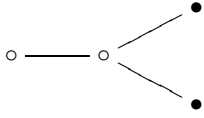
\xymatrix@R=0.8em{
&&\bullet\\
\circ\ar@{-}[r]&\circ\ar@{-}[ru]\ar@{-}[rd]&\\
&&\bullet\\
}$$

\item
$Y$ has two singular points of types ${A}_1$, $Y'$ is a blow up of a point and two infinitely near points, $\rho(Y)=2$, and $Y'$ has $2$ $(-1)$-curves.
$$
\xymatrix@R=0.8em{
\\
\circ\ar@{-}[r]&\bullet\ar@{-}[r]&\circ\ar@{-}[r]&\bullet\\
}
$$

\item
$Y$ has two singular points of type ${A}_1$ and ${A}_2$, $Y'$ is a blow up of three infinitely near points, $\rho(Y)=1$, and $Y'$ has $1$ $(-1)$-curve.
$$
\xymatrix@R=0.8em{
\\
\circ\ar@{-}[r]&\circ\ar@{-}[r]&\bullet\ar@{-}[r]&\circ\\
}
$$
\end{enumerate}

\

\

Now, we consider cubic del Pezzo surfaces with du Val singularities.

\begin{lemma}
\label{lem-singular-cubic}
Let $F$ be a del Pezzo surface of degree $3$ with du Val singularities. Assume that $F$ admits a conic bundle structure $\pi\colon F\to \mathbb{P}^1$ with a multiple fiber $C'=2C$. Assume that the singularities of $F$ along $C$ are of type $A_1$. Then
\begin{enumerate}
\item
there are precisely two singular points of type $A_1$ contained in $C$;
\item
the possibilities for the singularities of $F$ are as follows: $2A_1$, $3A_1$, $4A_1$, $2A_1A_2$, $2A_1A_3$;
\item
there exists a unique $(-1)$-curve intersecting $C$ outside of singular points.
\end{enumerate}
\end{lemma}
\begin{proof}
The first claim follows from the classification of singular points along a multiple fiber of a conic bundle on del Pezzo surfaces with du Val singularities, see \cite[Lemma 1.5]{Zh} or  \cite[Lemma 2.12]{ChPr21}. However, for the convenience of the reader, we give an independent proof of the first claim.

Consider the minimal resolution $\phi\colon \widetilde{F}\to F$. Then $\widetilde{F}$ is a smooth weak del Pezzo surface, which means that $-K_{\widetilde{F}}$ is big and nef. Let $\widetilde{\pi}=\pi \circ \phi$ be the induced conic bundle structure on $\widetilde{F}$. First of all note that since $C'=2C$ is a multiple fiber of $\pi$, $F$ should be singular at some point contained in $C$.
 Let $\widetilde{C}$ be the strict transform of $C$ on $\widetilde{F}$. Then the integer $\widetilde{C}^2=-k/2$ where $k$ is the number of singular points (which by assumption are of type $A_1$) contained in $C$. Since $\widetilde{F}$ is a smooth weak del Pezzo surface, we have $k\in\{2,4\}$. However, the case $k=4$ is not realized because this would contradicts to the fact that $(\pi^*C)^2=0$.




The second claim follows from the classification of the configurations of singular points on a cubic del Pezzo surface with du Val singularities. Assuming that $F$ has two points of type $A_1$, we have the following possibilities:
$2A_1$, $3A_1$, $4A_1$, $2A_1A_2$, $2A_1A_3$ (cf. \cite[]{BrW79}).


The last claim of the lemma follows from considering the diagrams of intersection of the $(-1)$-curves and the $(-2)$-curves on $\widetilde{F}$ as shown in \cite{De23}.
\end{proof}

\section{Fano threefold of rank $3$ and degree $14$}
\label{sec-prelim}
Throughout the paper, we shall use the following notation. Let $X$ be a smooth Fano threefold with Picard rank $3$ and of degree $14$. Let $Q=\pp_S(\oo\oplus\oo(-H)\oplus\oo(-H))$, where $H$ is the divisor of degree $(1,1)$ on $S=\mathbb{P}^1\times\mathbb{P}^1$.
Denote by $p\colon Q\to\mathbb{P}^1\times\mathbb{P}^1$ the natural projection.
Then $X$ can be realized as a divisor in the linear system $|L^{\otimes2}\otimes p^*\oo(2,3)|$ on the $Q$ such that $X\cap Y$ is irreducible, where $L$ is the tautological bundle,  $Y\in |L|$.  Let $[s_0:s_1:t_0:t_1:u_0:u_1:u_2]$ be homogeneous coordinates on the fourfold $Q$ such that
\[
wt(s_0)=(1,0,0),\ \
wt(s_1)=(1,0,0),\ \
wt(t_0)=(0,1,0),\ \
wt(t_1)=(0,1,0),
\]
\[
wt(u_0)=(0,0,1),\ \
wt(u_1)=(1,1,1),\ \
wt(u_2)=(1,1,1).
\]
The projection $\pi$ is given by the formula
$
[s_0:s_1:t_0:t_1:u_0:u_1:u_2]\mapsto [s_0:s_1:t_0:t_1]
$
where we consider $[s_0:s_1:t_0:t_1]$ as homogeneous coordinates on $\mathbb{P}^1\times\mathbb{P}^1$.  Since $X$ is the divisor in the linear system  $|L^{\otimes 2}\otimes p^*\mathcal{O}(2,3)|$, it is given by the following equation of weight $(2, 3, 2)$:
\[
f_1(t_0,t_1)u_1^2+f_2(t_0,t_1)u_2^2+f_3(t_0,t_1)u_1u_2+g_1(s_0,s_1,t_0,t_1)u_0u_1+g_2(s_0,s_1,t_0,t_1)u_0u_2+h(s_0,s_1,t_0,t_1)u_0^2=0
\] where $f_1,f_2,f_3$ are homogeneous polynomials of degree one  in $t_j$, $g_1,g_2$ are homogeneous polynomials that have degree one in $s_0,s_1$ and degree two in $t_0,t_1$, $h$ is a homogeneous polynomial that has degree two in $s_0,s_1$ and degree three in $t_0,t_1$.

We have two fibrations $\pi_1\colon X\rightarrow\pp^1$ and $\pi_2\colon X\rightarrow\pp^1$ into del Pezzo surfaces. Let $F_1$ and $F_2$ be general fibers of the del
Pezzo fibrations $\pi_1$ and $\pi_2$, respectively. Note that $F_1$ is a del Pezzo surface of degree six, $F_2$ is a del Pezzo surface of degree three. By \cite{MM83} there exists a divisor $D\simeq\pp^1\times\pp^1$ such that $\oo_D(D)=\oo_{\pp^1\times\pp^1}(-1,-1)$.
We see that $-K_X\sim D+F_1+2F_2$.

By \cite[p. 71]{Mat95}, see also the erratum \cite[p. 42]{Mat23}, the cone of effective divisors $\overline{\mathrm{Eff}}(X)$ is generated by the surfaces $D,F_1,F_2$. The Mori cone is generated by the two rulings $L_1$ and $L_2$ on $D\simeq\mathbb{P}^1\times\mathbb{P}^1$ (chosen in such a way that $F_i=\pi^{-1}(\pi(L_i)$), and by a general fiber $C$ of a conic bundle $\pi\colon X\to\mathbb{P}^1\times\mathbb{P}^1$. Note that $D$ is a bisection of the conic bundle $\pi$. The intersection theory on $X$ is as follows:
\[
F_1|_D = 2L_1, \quad \quad \quad F_2|_D = L_2, \quad \quad \quad F_1\cdot F_2 = C, \quad \quad \quad F_i^2=0,
\]
\[
D^3 = D|_D^2 = 2, \quad \quad \quad D\cdot F_1\cdot F_2 = 2, \quad \quad \quad D^2\cdot F_1 = -2, \quad \quad \quad D^2 \cdot F_2 = -1,
\]
\[
D\cdot L_i = -1, \quad \quad \quad \quad F_i\cdot L_i = 0, \quad \quad \quad \quad F_1\cdot L_2 = 2, \quad \quad \quad \quad F_2\cdot L_1 = 1.
\]
We have the following diagram of contractions on $X$:
\begin{equation}
\label{diagram-contractions}
\begin{tikzcd}
V \ar[dd, "\mu_1"'] \arrow[rrrr,bend left,dashed,"\gamma"]  \ar[rr, "\psi_1"] & & Y & & V'  \ar[ll, "\psi_2"'] \ar[dd, "\mu_2"] \\
.& & \ar[lld, "\pi_1"'] \ar[llu, "\phi_1"] X \ar[rrd, "\pi_2"] \ar[u, "\phi"]  \ar[rru, "\phi_2"'] \ar[d, "\pi"'] & & \\
\mathbb{P}^1 & & \ar[ll, "\mathrm{pr}_1"] \mathbb{P}^1\times \mathbb{P}^1 \ar[rr, "\mathrm{pr}_2"'] & & \mathbb{P}^1
\end{tikzcd}\vspace{0.1cm}
\end{equation}
Here $\phi$ is a contraction of a divisor $D$ to an ordinary double point on singular Fano threefold $Y$ (see \cite[2.3.8]{Ta22}, \cite[\textnumero9]{CKGSh23}), $\phi_1$ and $\phi_2$ are contractions of $D$ along two different rulings so that $V$ and $V'$ are smooth threefolds, $\mu_1$ is a fibration into del Pezzo surfaces of degree $8$, and $\mu_2$ is a fibration into del Pezzo surfaces of degree $4$, $\psi_1$ and $\psi_2$ are two small contractions with the exceptional locus $\mathbb{P}^1$, $\gamma$ is the Atiyah flop.


\begin{lemma}
\label{Lemm1}
Let $P\in X$ be a point.
Let $F_1,F_2$ be the fibers of $\pi_1$ and $\pi_2$, respectively, that contain~$P$. Assume that $F_i$ is singular at $P$. Then another surface $F_j$ is smooth at $P$.
\end{lemma}
\begin{proof}
Let $Q$ be the $\pp^2$-bundle over $S=\mathbb{P}^1\times\mathbb{P}^1$ such that $X$ is a divisor in the linear system $|L^{\otimes2}\otimes\oo(2,3)|$ on $Q$. We can pick a local chart $\aaa^4\subset Q$ such that $P\in\aaa^4$. Put $X'=\aaa^4\cap X$, $F'_1=\aaa^4\cap F_1$, $F'_2=\aaa^4\cap F_2$. Moreover, we may choose coordinates $(x,y,z,t)$ on $\aaa^4$ such that $X'$ is given by $f(x,y,z,t)=0$, $P$ is $(0,0,0,0)\in\aaa^4$, $F'_1$ is given by $f(x,y,z,t)=0$ and $t=0$, $F'_2$ is given by $f(x,y,z,t)=0$ and $z=0$. Assume that $F'_1$ is singular at $P$. Then the linear term in $f(x,y,z,0)$ vanishes. Assume that $F'_2$ also is singular at $P$. Then the linear term in $f(x,y,0,t)$ vanishes as well. Hence, the linear term in $f(x,y,z,t)$ also vanishes. So, $X'$ is singular at $P$, which is a contradiction.
\end{proof}

\begin{lemma}
\label{Lemm2}
Let $F_2$ be a fiber of $\pi_2$. Then $F_2$ is normal.
\end{lemma}
\begin{proof}
It is enough to prove that $F_2$ is smooth at codimension $1$.
Assume that there exists a fiber $F_2$ of $\pi_2$ such that $F_2$ has a curve of singularities $B$.
We claim that $B\cdot F_1>0$. Indeed, if $B\cdot F_1=0$, then $B$ is a set-theoretic fiber of $\pi\colon X\rightarrow\pp^1\times\pp^1$. Moreover, intersecting $F_1$ with $F_2$ we see that $B$ is a multiple fiber of $\pi$.
Thus, $B$ is a multiple fiber of $\pi_2$ as well. So, $B$ contains a singular point of $F_1$, which contradicts Lemma \ref{Lemm1}. This shows that $B\cdot F_1>0$. So, $B$ is a horizontal curve of conic bundle, i.e. $\pi(B)$ is a curve of type $(1,0)$ on $\pp^1\times\pp^1$. Then every fiber of $\pi$ in $F_2$ is singular.

Let $\Delta\subset \pp^1\times\pp^1$ be the discriminant curve of $\pi$. Note that $\Delta=\Delta_1\cup\Delta_2$, where $\Delta_1$ is a divisor of type $(1,0)$, $\Delta_1$ is a divisor of type $(1,5)$. Since $\Delta_1\cdot\Delta_2=5$, and $\Delta_1$ intersects $\Delta_2$ transversally, we see that there exist five multiple fibers of $\pi$ on $F_2$. On the other hand, since a general fiber of $\pi_2$ meets $D\simeq \mathbb{P}^1\times\mathbb{P}^1$ by an irreducible smooth rational curve which is a ruling on $D$, we see that $F_2\cap D$ is isomorphic to $\mathbb{P}^1$.

So, there exists a two-fold covering $\psi\colon F_2\cap D\longrightarrow\pp^1$. Let us denote the points of intersection $\Delta_1\cap\Delta_2$ as $P_1,\ldots,P_5$. We see that fibers of $\pi$ over $P_1,\ldots,P_5$ are multiple fibers $C_1,\ldots,C_5$. Then $C_i=2C'_i$. Since a general fiber of $\pi$ meets $D$ in two points, we see that $C'_i\cdot D=1$. Then $\psi$ has at least five ramification points. This contradicts to the Hurwitz formula applied to the map $\psi$. The proof is completed.
\end{proof}

\begin{lemma}
\label{Lemm2a}
Let $F_1$ be a fiber of $\pi_1$. Then either $F_1$ is normal or the curve of singularities of $F_1$ coincides with $F_1\cap D$. 
\end{lemma}
\begin{proof}
It is enough to prove that $F_1$ is smooth at codimension $1$ or the curve of singularities on $F_1$ is $F_1\cap D$.
Assume that there exists a fiber $F_1$ of $\pi_1$ such that $F_1$ has a curve of singularities $B$.
Arguing as in the proof of Lemma Lemma \ref{Lemm2} we have that $B\cdot F_2>0$.


Let $\Delta$ be the discriminant curve on $\pp^1\times\pp^1$. Note that $\Delta=\Delta_1\cup\Delta_2$, where $\Delta_1$ is a divisor of type $(0,1)$, $\Delta_1$ is a divisor of type $(2,4)$. Since $\Delta_1\cdot\Delta_2=2$, we see that $\pi$ has two multiple fibers on $F_1$, say $C_1$, $C_2$. Note that $C_i=2C'_i$. Then $C'_i\cdot D=1$, i.e. $C_i$ meets $D$ in one point. Since $F_1|_D \sim 2L_1$, we see that $F_1\cap D$ consists of an irreducible non-reduced curve. Arguing as in the proof of Lemma \ref{Lemm2}, we have that every general fiber $C$ of $\pi$ on $F_1$ consists of two lines $l_1, l_2$. Since $C\cdot D=2$, we see that $l_1\cap L_1=l_2\cap L_1$ consists of one point. Then $L_1$ is the curve of singularities of $F_1$.
\end{proof}

\begin{lemma}
\label{Lemm3}
Let $F_1,F_2$ be fibers of $\pi_1$ and $\pi_2$. Assume that $F_1$ is normal. Then $F_1,F_2$ are Gorenstein del Pezzo surfaces with at worst du Val singularities.
\end{lemma}
\begin{proof}
We have that $F_1$ is normal by assumption, and $F_2$ is normal by Lemma \ref{Lemm2}. The fact that $F_i$ are Gorenstein del Pezzo surfaces follows from adjunction formula on $X$. From \cite{HW81} it follows that $F_i$ have at worst log canonical singularities, and if $F_i$ is strictly log canonical, then it has a simple elliptic singularity and $F_i$ is non-rational.
However, since $F_i$ has a structure of a conic bundle over a rational curve, we conclude that $F_i$ is rational. Thus, $F_i$ has at worst du Val singularities.
\end{proof}

\begin{lemma}
\label{lem-F1-smooth}
Assume that $P\not\in D$ and $P$ is a singular point of $F_2$.
Assume that the fiber of $\pi$ that contains $P$ is not multiple. Then $F_1$ is smooth where $F_1$ is a fiber of $\pi_1$ that contains $P$.
\end{lemma}
\begin{proof}
By Lemma \ref{Lemm1} we see that $F_1$ is smooth at $P$. Note that the fiber of $\pi$ that contains $P$ consists of two curves. Since the intersection points of these curves is a smooth point of $F_1$, by the proof of Lemma \ref{Lemm2a} we see that $F_1$ is normal. Hence by Lemma \ref{Lemm3} we have that $F_1$ is a Gorenstein del Pezzo surface with at worst du Val singularities.

Note that $D\cap F_1=E_1\cup E_2$ where $E_1,E_2$ are disjoint $(-1)$-curves. Indeed, this follows from the equations $D|_{F_1}^2=-2$ and $-K_{F_1}\cdot D|_{F_1} = 2$.
From the classification of sextic du Val del Pezzo surfaces as in section \ref{subsect-sextic} it follows that the only possible case when $F_1$ is singular is when there are four $(-1)$-curves on $F_1$ and one singular point of type $A_1$. 
Let $C$ be a fiber of $\pi\colon X\to \mathbb{P}^1\times\mathbb{P}^1$ that passes through $P$. Since $P\in F_2$ is a singular point, the fiber of the conic bundle $\pi|_{F_2}\colon F_2\to \mathbb{P}^1$ is singular. Since by assumption it is not multiple, we have that $C$ is reducible. Then on $F_1$ we have $C=E_3+E_4$, and $P=E_3\cap E_4$ which is a singular point on $F_1$. However, this contradicts to Lemma \ref{Lemm1}. This shows that $F_1$ is smooth.
\end{proof}

\section{K-stability and Abban-Zhuang theory}
We briefly recall some of the definitions in the theory of K-stability. For more details, see a survey \cite{Xu21} and references therein.

\subsection{Discrepancies and thresholds}
Let $X$ be a Fano variety, and let $f\colon Y\to X$ be a proper birational morphism from a normal variety $Y$. For a prime divisor $E$ on $Y$, a \emph{log discrepancy} of $E$ with respect to $X$ is defined as
\[
A_{X}(E) = 1 + \mathrm{coeff}_E ( K_Y - f^*(K_X)).
\]
Put $L=-K_X$. By a \emph{pseudo-effective threshold} of $E$ with respect to a Fano $X$ we mean the number
\[
\tau_{X}(E) = \sup\{ x\in\mathbb{R}_{\geq0}\colon f^*L - xE\ \text{is pseudo-effective}\}.
\]
Similarly, we define the \emph{nef threshold} of $E$ with respect to a Fano $X$:
\[
\epsilon_{X}(E) = \sup\{ x\in\mathbb{R}_{\geq0}\colon f^*L - xE\ \text{is nef}\}.
\]
The \emph{expected vanishing order} of $E$ with respect to a Fano $X$ is
\[
S_{X}(E) = \frac{1}{\mathrm{vol}(L)} \int_{0}^{\infty}{\mathrm{vol}( f^*L - xE )dx},
\]
where $\mathrm{vol}$ is the volume function, see \cite{Laz04}. The \emph{beta-invariant} $\beta_{X}(E)$ of $E$ with respect to a Fano $X$ is defined as follows:
\[
\beta_{X}(E) = A_{X}(E) - S_{X}(E).
\]
Recall that the $\delta$-invariant of $E$ with respect to a Fano $X$ (resp., $\delta$-invariant of $E$ along $Z$ with respect to a Fano $X$) are defined as
\[
\delta(X) = \inf_{E/X} \frac{A_{X}(E)}{S_{X}(E)}, \ \ \ \ \delta_Z(X) = \inf_{E/X,\ Z\subset C(E)} \frac{A_{X}(E)}{S_{X}(E)}
\]
where $E$ runs through all prime divisors over $X$ (resp., $E$ runs through all prime divisors over $X$ whose center contains $Z$).

\begin{definition}[\cite{Li17}, \cite{Fu19}, \cite{Fu16}]
\label{def-k-stability}
A klt Fano $X$ is called
\begin{enumerate}
\item
\emph{divisorially semistable} (resp., \emph{divisorially stable}), if $\beta_{X}(E)\geq0$ (resp., $\beta_{X}(E)>0$) for any prime divisor $E$ on $X$. We say that $X$ is \emph{divisorially unstable} if it is not divisorially semistable,
\item
\emph{K-semistable} (resp., \emph{K-stable}) if $\beta_{X}(E)\geq0$ (resp., $\beta_{X}(E)>0$) for any prime divisor $E$ over $X$. We say that $X$ is \emph{K-unstable} if it is not K-semistable.
\end{enumerate}
\end{definition}


Now, we recall two propositions from Abban-Zhuang theory developed in \cite{AZ20}.

\begin{proposition}[{\cite[Corollary 1.7.26]{ChAll}}]
\label{Van1}
Let $X$ be a smooth Fano threefold, $Y\subset X$ be an irreducible normal surface that has at most du Val singularities, $Z\subset Y$ be an irreducible smooth curve. Then for any prime divisor $E$ over $X$ such that $C(E)=Z$ we have
\begin{equation}
\label{ineq-curve}
\frac{A_X(E)}{S_X(E)}\geq \min\left\{ \frac{1}{S_X(Y)}, \frac{1}{S(W^{Y,Z}_{\bullet,\bullet};Z)} \right\},
\end{equation}
where
\begin{multline*}
S(W^{Y,Z}_{\bullet,\bullet};Z)=\frac{3}{(-K_X)^3}\int\limits_0^{\infty}(P(u)^2\cdot Y)\cdot\ord_Z(N(u)|_Y)du
+\frac{3}{(-K_X)^3}\int\limits_0^{\infty}\int\limits_0^{\infty}\vol(P(u)|_Y-vZ)dv du.
\end{multline*}
Moreover, if the equality holds in \eqref{ineq-curve}, then $\frac{A_X(E)}{S_X(E)} = \frac{1}{S_X(Y)}$.
\end{proposition}


Let $P(u,v)$ be the positive part of the Zariski decomposition of $P(u)|_Y-vZ$, and $N(u,v)$ be the negative
part of the Zariski decomposition of this divisor. We can write $N(u)|_Y=dZ+N'_Y(u)$, where $Z\not\subset
\mathrm{Supp}(N'_Y(u))$ and $d=d(u)=\ord_Z(N(u)|_Y)$.

\begin{proposition}[{\cite[Theorem 1.7.30]{ChAll}}]
\label{Van2}
Let $X$ be a smooth Fano threefold, $Y\subset X$ be an irreducible normal surface that has at most du Val singularities, $Z\subset Y$ be an irreducible smooth curve such that the log pair $(Y,Z)$ has purely log
terminal singularities. Let $P$ be a point in the curve $Z$. Then
\begin{equation}
\delta_P(X)\geq\min\left\{\frac{1-\ord_P(\Delta_Z)}{S(W^{Y,Z}_{\bullet,\bullet,\bullet};P)},\frac{1}{S(V^Y_{\bullet,\bullet};Z)},\frac{1}{S_X(Y)}\right\},
\end{equation}
where $\Delta_Z$ is the different of the log pair $(Y,Z)$, and
\begin{multline*}
S(W^{Y,Z}_{\bullet,\bullet,\bullet};P)=\frac{3}{(-K_X)^3}\int\limits_0^{\infty}\int\limits_0^{\infty}(P(u,v)\cdot Z)^2dv du+\\
+\frac{6}{(-K_X)^3}\int\limits_0^{\infty}\int\limits_0^{\infty}(P(u,v)\cdot Z)\cdot\ord_P(N'_Y(u)|_Z+N(u,v)|_Z)dv du.
\end{multline*}
Moreover, if the inequality is an equality and there exists a prime divisor $E$ over
the threefold such that $C_X(E)=P$ and $\delta_P(X)=\frac{A(E)}{S(E)}$ then $\delta_P(X)=\frac{1}{S_X(Y)}$.
\end{proposition}

\begin{proposition}[{\cite[Theorem 1.7.30]{ChAll}},{\cite[Remark 1.7.32]{ChAll}}]
\label{Van3}
Let $X$ be a smooth Fano threefold, $Y\subset X$ be an irreducible normal surface that has at most du Val singularities, let $Q\in Y$ be a point in $Y$. $\epsilon\colon\tilde{Y}\rightarrow Y$ be the plt blowup of the point $Q$, and let $\tilde{Z}$ be the $\epsilon$-exceptional curve. Then
\begin{equation}
\delta_Q(X)\geq\min\left\{\min\limits_{P\in\tilde{Z}}\frac{1-\ord_P(\Delta_{\tilde{Z}})}{S(W^{Y,\tilde{Z}}_{\bullet,\bullet,\bullet};P)},\frac{A_Y(\tilde{Z})}{S(V^Y_{\bullet,\bullet};\tilde{Z})},\frac{1}{S_X(Y)}\right\},
\end{equation}
where $\Delta_{\tilde{Z}}$ is the different of the log pair $(\widetilde{Y},\tilde{Z})$.
\end{proposition}

\subsection{Applications of Abban-Zhuang theory}
\label{subsec-applications-az}
Let $P$ be a point in $X$. To prove that $X$ is K-stable, it is enough to show that $\delta_P(X)>1$.
We can estimate $\delta(P)$ as in \cite[Theorem 3.3]{AZ20} and Proposition \ref{Van3}.
Let $Y$ be a normal irreducible surface in $X$.
Then from \cite{AZ20} and \cite{ChAll} it follows that
\begin{equation}
\label{eq-applications}
\delta_P(X)\geq \min \left\{\frac{1}{S_X(Y)}, \delta_P(Y, W_{\bullet, \bullet}^{Y})\right\}
\end{equation}
for
\[
\delta_P(Y, W_{\bullet, \bullet}^{Y}) = \inf_{E/Y, P\in C_{Y}(E)} \frac{A_{Y}(E)}{S(W_{\bullet, \bullet}^{Y}, E)}
\]
where by Proposition \ref{Van2} one has
\begin{multline*}
S(W^{Y}_{\bullet,\bullet};E)=\frac{3}{(-K_X)^3}\int\limits_0^{\infty}(P(u)^2\cdot S)\cdot\ord_E(N(u)|_Y)du
+\frac{3}{(-K_X)^3}\int\limits_0^{\infty}\int\limits_0^{\infty}\vol(P(u)|_Y-vE)dv du
\end{multline*}
and the infimum is taken over all prime divisors $E$ over $Y$ whose centers on $Y$ contain~$P$.

\section{Divisorial stability}
\label{sec-divisorial-stability}
For the reader's convenience, we prove that $X$ is divisorially stable.
By {\cite[Lemma 9.5, Remark 9.6]{Fu16}} it is enough to consider only the divisors $L$ such that $-K_X-L$ is big. We claim that it is enough to consider only one divisor $L=F_2$. Write $L\sim aD+bF_1+cF_2$ where $a,b,c\geq 0$, because the cone of effective divisors on $X$ is generated by $D,F_1,F_2$. It is clear that $a=0$. Note that $-K_X-L$ is big, then $-K_X-L|_{F_i}$ are big as well for $i=1,2$ where $F_i$. We have $-K_X\sim D+F_1+2F_2$, so $-K_X|_{F_1}\sim D+2F_2|_{F_1}$ and $-K_X|_{F_2}\sim D+F_1|_{F_2}$. This shows that $b=0$ and $c=1$. We start with computing some Zariski decompositions.

\begin{proposition}
\label{ZD2}
Let $-K_X-uF_1=P(u)+N(u)$ be the Zariski decomposition. Then
\[
P(u)=
\begin{cases}D+(1-u)F_1+2F_2,\quad\quad\quad\quad\quad\quad\text{for}\ \ 0\leq u\leq\frac{1}{2}, \\
(2-2u)D+(1-u)F_1+2F_2,\quad\ \quad\, \text{for}\quad \frac{1}{2}\leq u\leq 1.
\end{cases}
\]
\[
N(u)=
\begin{cases}0,\quad\, \quad\quad\quad\quad\quad\text{for}\ \ 0\leq u\leq\frac{1}{2}, \\
(2u-1)D,\quad\ \quad\text{for}\quad \frac{1}{2}\leq u\leq 1.
\end{cases}
\]
\end{proposition}
\begin{proof}
In the above notation we have $(-K_X-uF_1)\cdot L_1 = 1$, $(-K_X-uF_1)\cdot L_2 = 1-2u$, and $(-K_X-uF_1)\cdot C = 2$.
Hence $-K_X-uF_1$ is ample for $0\leq u<\frac{1}{2}$.
Then for $\frac{1}{2}\leq u\leq 1$, we have
$P(u)=(2-2u)D+(1-u)F_1+2F_2$ and $N(u)=(2u-1)D$.
\end{proof}

\begin{proposition}
\label{ZD}
Let $-K_X-uF_2=P(u)+N(u)$ be the Zariski decomposition. Then
\[
P(u)=
\begin{cases}D+F_1+(2-u)F_2,\quad\quad\quad\quad\quad\quad\text{for}\ \ 0\leq u\leq1, \\
(2-u)D+F_1+(2-u)F_2,\quad \quad \quad\text{for}\ \ 1\leq u\leq2.
\end{cases}
\]
\[
N(u)=
\begin{cases}0,\quad\ \quad\quad\quad\quad\quad\text{for}\ \ 0\leq u\leq1, \\
(u-1)D,\quad \quad \quad\text{for}\ \ 1\leq u\leq2.
\end{cases}
\]
\end{proposition}
\begin{proof}
In the above notation we have $(-K_X-uF_2)\cdot L_1 = 1-u$, $(-K_X-uF_2)\cdot L_2 = 1$, and $(-K_X-uF_2)\cdot C = 2$. Then $-K_X-uF_2$ is ample for $0\leq u<1$. 
For $1\leq u\leq 2$, we have
$P(u)=(2-u)D+F_1+(2-u)F_2$ and $N(u)=(u-1)D$.
\end{proof}
Now we compute $\beta_X(F_2)=1-S_X(F_2)$. Start with
\begin{multline*}
S_X(F_2) = \frac{1}{(-K_X)^3}\int\limits_0^{2}\vol(-K_X-uF_2)dt=\frac{1}{14}\int\limits_0^1(D+F_1+(2-u)F_2)^3du+\\
+\frac{1}{14}\int\limits_1^2((2-u)D+F_1+(2-u)F_2)^3du=\\
=\frac{1}{14}\int\limits_0^1(14-9u)du+\frac{1}{14}\int\limits_1^2(6(2-u)^2-(2-u)^3)du=\frac{19}{28}+\frac{7}{56}=\frac{45}{56}<1.
\end{multline*}
So, we obtain $S_X(F_2)=\frac{45}{56}$, and hence $\beta_X(F_2) = 1-\frac{45}{56}>0$. Thus $X$ is divisorially stable.

\section{Computations}
\label{sec-center-point}
In this section, we work in the following setting. Assume that $X$ is a smooth threefold with Picard rank $3$ and of degree $14$, and that $P$ is a point in $X$. Let $F_1$ be the fiber of $\pi_1$ that contains $P$. Let $F_2$ be the fiber of $\pi_2$ that contains $P$.

\begin{lemma}
\label{DeltaD} Assume that $P\in D$. Then $\delta_P(X)\geq\frac{56}{45}$.
\end{lemma}

\begin{proof}
This prove is  similar to the proof of Lemma 5.68 in \cite{ChAll}. 
Put $Z\subset D$ is a divisor of type $(1,0)$ that contains $P$. Let $-K_X-uD=P(u)+N(u)$ be the Zariski decomposition. Note that $P(u)=-K_X-uD$, $N(u)=0$ for $0\leq u\leq 1$ and $P(u)=0$ for $u>1$.
We have \begin{multline*}
S(W^{D,Z}_{\bullet,\bullet};Z)=\frac{3}{(-K_X)^3}\int\limits_0^{\infty}(P(u)^2\cdot D)\cdot\ord_Z(N(u)|_D)du
+\frac{3}{(-K_X)^3}\int\limits_0^{\infty}\int\limits_0^{\infty}\vol(P(u)|_D-vZ)dv du=\\
=\frac{3}{14}\int\limits_0^{\infty}\int\limits_0^{\infty}\vol(((1-u)D+F_1+2F_2)|_D-vZ)=\frac{3}{14}\int\limits_0^{1}\int\limits_0^{u+1}R^2dvdu,
\end{multline*}
where $R$ is a divisor of type $(u-v+1,u+1)$. Then $R^2=2(u-v+1)(u+1)$. So, \[
S(W^{D,Z}_{\bullet,\bullet};Z)=\frac{3}{14}\int\limits_0^{1}\int\limits_0^{u+1}2(u-v+1)(u+1)dvdu=\frac{45}{56}.
\]

Also, we have
\begin{multline*}
S(W^{D,Z}_{\bullet,\bullet,\bullet};P)=\frac{3}{(-K_X)^3}\int\limits_0^{\infty}\int\limits_0^{\infty}(P(u,v)\cdot Z)^2dv du+\\
+\frac{6}{(-K_X)^3}\int\limits_0^{\infty}\int\limits_0^{\infty}(P(u,v)\cdot Z)\cdot\ord_P(N'_D(u)|_Z+N(u,v)|_Z)dv du,
\end{multline*} where $P(u,v)$ is the positive part of the Zariski decomposition of $P(u)|_D-vZ$, $N(u,v)$ is the negative
part of the Zariski decomposition of this divisor, $N'_D(u)=N(u)|_D-dZ$, where $Z\not\subset
\mathrm{Supp}(N'_D(u))$ and $d=d(u)=\ord_Z(N(u)|_D)$. Note that $N(u,v)=0$ for $0\leq v\leq u+1$ and $P(u,v)=0$ for $v>u+1$. Then \[
S(W^{D,Z}_{\bullet,\bullet,\bullet};P)=\frac{3}{14}\int\limits_0^{1}\int\limits_0^{u+1}(R\cdot Z)^2dv du=\frac{3}{14}\int\limits_0^{1}\int\limits_0^{u+1}(u+1)^2dv du=\frac{45}{56}.
\]
So, $\delta_P(X)\geq\frac{56}{45}$ (see Propositions \ref{Van1} and \ref{Van2}).
\end{proof}

\begin{lemma}
\label{DeltaSmF2} Assume that $P\not\in D$ and $F_2$ is a del Pezzo surface such that $F_2$ is smooth along $F_1\cap F_2$. Then $\delta_P(X)>1$.
\end{lemma}

\begin{proof}
This proof is similar to the proof of Lemma 5.69 in \cite{ChAll}. 
Put $Z=F_1\cap F_2$. Assume that $Z$ is an irreducible curve. Note that $Z$ is a $(0)$-curve on $F_2$. Moreover, \[-K_{F_2}=-K_X|_{F_2}=Z+E,\] where $E=D\cap F_2$ and $E\cdot Z=2$. Note that $E^2=-1$ and $E$ is a smooth rational curve, since $E=D\cap F_2$ is a ruling on $D\simeq \mathbb{P}^1\times\mathbb{P}^1$. Hence, $E$ is a $(-1)$-curve on $F_2$.

Recall that we denote by $P(u,v)$ the positive part of the Zariski decomposition of $P(u)|_{F_2}-vZ$, and by $N(u,v)$ the negative part of the Zariski decomposition of this divisor, where $P(u)$ is given by Proposition \ref{ZD}. To compute the Zariski decomposition, note that for any $(-1)$-curve $E'$ on $F_2$ different from $E$ we have that $E'$ intersects $Z$ in at most one point. Indeed, this follows from the inequality
\[
1=(-K_{F_2})E' = (Z+E)E' \geq Z E'.
\]
According to Remark \ref{lem-duVal_contractions}, any birational contraction on $F_2$ is a contraction of a $(-1)$-curve. Consequently, we obtain
\[
P(u,v) = \begin{cases}(1-v)Z+E, \quad \quad \text{for}\quad \quad 0\leq u\leq 1,\ \ 0\leq v\leq 1/2, \\
(1-v)Z+2(1-v)E, \quad \quad \text{for}\quad \quad 0\leq u\leq 1,\ \ 1/2< v\leq 1, \\
(1-v)Z+(2-u)E, \quad \quad \text{for}\quad \quad 1\leq u\leq 2,\ \ 0\leq v\leq u/2, \\
(1-v)Z+2(1-v)E, \quad \quad \text{for}\quad \quad 1\leq u\leq 2,\ \ u/2< v\leq 1.
\end{cases}
\]
and
\[
N(u,v)=\begin{cases}0\quad \quad \text{for}\quad \quad 0\leq u\leq 1,\ \ 0\leq v\leq 1/2, \\
(2v-1)E, \quad \quad \text{for}\quad \quad 0\leq u\leq 1,\ \ 1/2< v\leq 1, \\
0\quad \quad \text{for}\quad \quad 1\leq u\leq 2,\ \ 0\leq v\leq u/2, \\
(2v-u)E, \quad \quad \text{for}\quad \quad 1\leq u\leq 2,\ \ u/2< v\leq 1.
\end{cases}
\]

Compute
\begin{multline*}
S(W^{F_2}_{\bullet,\bullet};Z)=\frac{3}{(-K_X)^3}\int\limits_0^{\infty}(P(u)^2\cdot F_2)\cdot\ord_Z(N(u)|_{F_2})du
+\frac{3}{(-K_X)^3}\int\limits_0^{\infty}\int\limits_0^{\infty}\vol(P(u)|_{F_2}-vZ)dv du=\\
=\frac{3}{14}\int\limits_0^{1}\int\limits_0^{\infty}\vol((1-v)Z+E)dv du+\frac{3}{14}\int\limits_1^{2}\int\limits_0^{\infty}\vol((1-v)Z+(2-u)E)dvdu=\\
=\frac{3}{14}\int\limits_0^{1}\int\limits_0^{\frac{1}{2}}((1-v)Z+E)^2dv du+\frac{3}{14}\int\limits_0^{1}\int\limits_{\frac{1}{2}}^{1}((1-v)Z+2(1-v)E)^2dv du+\\
+\frac{3}{14}\int\limits_1^{2}\int\limits_0^{\frac{u}{2}}((1-v)Z+(2-u)E)^2dvdu+\frac{3}{14}\int\limits_1^{2}\int\limits_{\frac{u}{2}}^{1}((1-v)Z+2(1-v)E)^2dvdu=\\
=\frac{3}{14}\int\limits_0^{1}\int\limits_0^{\frac{1}{2}}(4(1-v)-1)dvdu+\frac{3}{14}\int\limits_0^{1}\int\limits_{\frac{1}{2}}^{1}(4(1-v)^2)dvdu+\frac{3}{14}\int\limits_1^{2}\int\limits_0^{\frac{u}{2}}(4(1-v)(2-u)-(2-u)^2)dvdu+\\
+\frac{3}{14}\int\limits_1^{2}\int\limits_{\frac{u}{2}}^{1}4(1-v)^2dvdu=\frac{3}{14}\left(1+\frac{1}{6}+\frac{2}{3}+\frac{1}{24}\right)=\frac{135}{336}<1.
\end{multline*}

Also, we have
\begin{multline*}
S(W^{F_2,Z}_{\bullet,\bullet,\bullet};P)=\frac{3}{(-K_X)^3}\int\limits_0^{\infty}\int\limits_0^{\infty}(P(u,v)\cdot Z)^2dv du+\\
+\frac{6}{(-K_X)^3}\int\limits_0^{\infty}\int\limits_0^{\infty}(P(u,v)\cdot Z)\cdot\ord_P(N'_{F_2}(u)|_Z+N(u,v)|_Z)dv du=\\
=\frac{3}{14}\int\limits_0^{1}\int\limits_0^{\frac{1}{2}}(((1-v)Z+E)\cdot Z)^2dv du+\frac{3}{14}\int\limits_0^{1}\int\limits_{\frac{1}{2}}^{1}(((1-v)Z+2(1-v)E)\cdot Z)^2dv du+\\
+\frac{3}{14}\int\limits_1^{2}\int\limits_0^{\frac{u}{2}}(((1-v)Z+(2-u)E)\cdot Z)^2dvdu+\frac{3}{14}\int\limits_1^{2}\int\limits_{\frac{u}{2}}^{1}(((1-v)Z+2(1-v)E)\cdot Z)^2dvdu=\\
=\frac{3}{14}\left(\int\limits_0^{1}\int\limits_0^{\frac{1}{2}}4dvdu+\int\limits_0^{1}\int\limits_{\frac{1}{2}}^{1}(16(1-v)^2)dvdu+\int\limits_1^{2}\int\limits_0^{\frac{u}{2}}(4(2-u)^2)dvdu+\int\limits_1^{2}\int\limits_{\frac{u}{2}}^{1}16(1-v)^2dvdu\right)=
\end{multline*}
\[=\frac{3}{14}\left(2+\frac{2}{3}+\frac{5}{6}+\frac{1}{6}\right)=\frac{11}{14}<1.\]
So, $\delta_P(X)>1$ (see Propositions \ref{Van1} and \ref{Van2}).

Assume that $F_1\cap F_2$ is reducible. We see that $F_1\cap F_2$ consists of two curves $E_1$ and $E_2$. Moreover, \[-K_{F_2}=-K_X|_{F_2}=E_1+E_2+E,\] where $E$ is a $(-1)$-curve such that $E=D\cap F_2$ and $E\cdot E_1=E\cdot E_2=E_1\cdot E_2=1$. We may assume that $P\in E_1$. We have \begin{multline*}
S(W^{F_2}_{\bullet,\bullet};E_1)=\frac{3}{(-K_X)^3}\int\limits_0^{\infty}(P(u)^2\cdot F_2)\cdot\ord_{E_1}(N(u)|_{F_2})du
+\frac{3}{(-K_X)^3}\int\limits_0^{\infty}\int\limits_0^{\infty}\vol(P(u)|_{F_2}-vE_1)dv du=\\
=\frac{3}{14}\int\limits_0^{1}\int\limits_0^{\infty}\vol((1-v)E_1+E_2+E)dv du+\frac{3}{14}\int\limits_1^{2}\int\limits_0^{\infty}\vol((1-v)E_1+E_2+(2-u)E)dvdu=\\
=\frac{3}{14}\int\limits_0^{1}\int\limits_0^{1}((1-v)E_1+E_2+E)^2dv du+\frac{3}{14}\int\limits_1^{2}\int\limits_0^{2-u}((1-v)E_1+E_2+(2-u)E)^2dvdu+\\
+\frac{3}{14}\int\limits_1^{2}\int\limits_{2-u}^{1}((1-v)E_1+(3-u-v)E_2+(2-u)E)^2dvdu=\\
=\frac{3}{14}\int\limits_0^{1}\int\limits_0^{1}(-(1-v)^2+4(1-v))dv du+\frac{3}{14}\int\limits_1^{2}\int\limits_0^{2-u}(-(1-v)^2-(2-u)^2-1+2(1-v)(2-u)+2(1-v)+2(2-u))dvdu+\\
+\frac{3}{14}\int\limits_1^{2}\int\limits_{2-u}^{1}(-(1-v)^2-(3-u-v)^2-(2-u)^2+2(1-v)(3-u-v)+2(1-v)(2-u)+2(3-u-v)(2-u))dvdu=\\
=\frac{3}{14}\left(\frac{5}{3}+\frac{3}{4}+\frac{1}{6}\right)=\frac{31}{56}<1.
\end{multline*}

Note that \begin{multline*}
\frac{3}{(-K_X)^3}\int\limits_0^{\infty}\int\limits_0^{\infty}(P(u,v)\cdot E_1)^2dv du=\frac{3}{14}\int\limits_0^{1}\int\limits_0^{1}(((1-v)E_1+E_2+E)\cdot E_1)^2dv du+\\
+\frac{3}{14}\int\limits_1^{2}\int\limits_0^{2-u}(((1-v)E_1+E_2+(2-u)E)\cdot E_1)^2dvdu+\frac{3}{14}\int\limits_1^{2}\int\limits_{2-u}^{1}(((1-v)E_1+(3-u-v)E_2+(2-u)E)\cdot E_1)^2dvdu=\\
=\frac{3}{14}\left(\int\limits_0^{1}\int\limits_0^{1}(v+1)^2dvdu+\int\limits_1^{2}\int\limits_0^{2-u}(2+v-u)^2dvdu+\int\limits_1^{2}\int\limits_{2-u}^{1}(4-2u)^2dvdu\right)=\frac{3}{14}\left(\frac{7}{3}+\frac{7}{12}+\frac{1}{3}\right)=\frac{39}{56}.
\end{multline*}

Assume that $P\neq E_1\cap E_2$. Then $S(W^{F_2,E_1}_{\bullet,\bullet,\bullet};P)=\frac{39}{56}<1$. So, $\delta_P(X)>1$ (see Propositions \ref{Van1} and \ref{Van2}).

Assume that $P=E_1\cap E_2$. Then \begin{multline*}
S(W^{F_2,E_1}_{\bullet,\bullet,\bullet};P)=\frac{39}{56}+\frac{6}{(-K_X)^3}\int\limits_0^{\infty}\int\limits_0^{\infty}(P(u,v)\cdot E_1)\cdot\ord_P(N'_{F_2}(u)|_{E_1}+N(u,v)|_{E_1})dv du=\\
\frac{39}{56}+\frac{6}{14}\int\limits_1^{2}\int\limits_{2-u}^{1}(((1-v)E_1+(3-u-v)E_2+(2-u)E)\cdot E_1)(u+v-2)E_2\cdot E_1dvdu=\\
=\frac{39}{56}+\frac{6}{14}\int\limits_1^{2}\int\limits_{2-u}^{1}(4-2u)(u+v-2)dvdu=\frac{39}{56}+\frac{1}{28}=\frac{41}{56}<1.
\end{multline*}
So, $\delta_P(X)>1$ (see Propositions \ref{Van1} and \ref{Van2}).
\end{proof}

\begin{lemma}
\label{DeltaSmF1} Assume that $P\not\in D$ and $F_1\cap F_2$ has a singular point of $F_2$.
Assume that the fiber of $\pi$ that contains $P$ is not multiple.
Then $\delta_P(X)\geq 1$.
\end{lemma}

Unfortunately, there was a mistake in the proof of this Lemma in the published version of the paper. We correct this mistake.
\begin{proof}
By Lemma \ref{lem-F1-smooth} we have that $F_1$ is smooth where $F_1$ is a fiber of $\pi_1$ that contains $P$. As in the proof of the lemma, we see that the intersection $D\cap F_1$ is a disjoint union of two $(-1)$-curves $E_1$ and $E_2$.
Also note that since $F_1\cap F_2$ has a singular point on $F_2$, we have that $F_1\cap F_2$ is a union of two $(-1)$-curves $E_3\cup E_4$ on $F_1$. Note that for a conic bundle $\pi|_{F_1}\colon F_1\rightarrow\pp^1$ we have that $E_1$ and $E_2$ are its sections. We may assume that $E_1\cdot E_3=E_2\cdot E_4=1$ and $E_2\cdot E_3=E_1\cdot E_4=0$.

Assume that $P\in E_3$, $P\not\in E_4$. Let $P(u)$ be the positive part of the Zariski decomposition of $-K_X-uF_1$. According to Proposition \ref{ZD2} we have
$P(u)=D+(1-u)F_1+2F_2$ and $N(u)=0$ for $0\leq u\leq\frac{1}{2}$, and $P(u)=(2-2u)D+(1-u)F_1+2F_2$ and $N(u)=(2u-1)D$ for $\frac{1}{2}\leq u\leq 1$. Note that $P(u)|_{F_1} = -K_{F_1}$ for $0\leq u\leq\frac{1}{2}$. Put $Z=E_3$.  We obtain
\[
P(u,v) = \begin{cases}E_1+E_2+(2-v)E_3+2E_4, \quad \quad \text{for}\quad \quad 0\leq u\leq \frac{1}{2},\ \ 0\leq v\leq 1, \\
(2-v)E_1+E_2+(2-v)E_3+(3-v)E_4, \quad \quad \text{for}\quad \quad 0\leq u\leq \frac{1}{2},\ \ 1\leq v\leq 2, \\
(2-2u)(E_1+E_2)+(2-v)E_3+2E_4, \quad \quad \text{for}\quad \quad \frac{1}{2}\leq u\leq 1,\ \ 0\leq v\leq 2-2u, \\
(2-2u)(E_1+E_2)+(2-v)E_3+(4-2u-v)E_4, \quad \quad \text{for}\quad \quad \frac{1}{2}\leq u\leq 1,\ \ 2-2u\leq v\leq 2u, \\
(2-v)E_1+(2-2u)E_2+(2-v)E_3+(4-2u-v)E_4, \quad \quad \text{for}\quad \quad \frac{1}{2}\leq u\leq 1,\ \ 2u\leq v\leq 2.
\end{cases}
\]
and
\[
N(u,v)=\begin{cases}0\quad \quad \text{for}\quad \quad 0\leq u\leq \frac{1}{2},\ \ 0\leq v\leq 1, \\
(v-1)(E_1+E_4), \quad \quad \text{for}\quad \quad 0\leq u\leq \frac{1}{2},\ \ 1< v\leq 2, \\
0\quad \quad \text{for}\quad \quad \frac{1}{2}\leq u\leq 1,\ \ 0\leq v\leq 2-2u, \\
(2u+v-2)E_4, \quad \quad \text{for}\quad \quad \frac{1}{2}\leq u\leq 1,\ \ 2-2u\leq v\leq 2u\\
(v-2u)E_1+(2u+v-2)E_4, \quad \quad \text{for}\quad \quad \frac{1}{2}\leq u\leq 1,\ \ 2u< v\leq 2.
\end{cases}
\]

Then \begin{multline*}
S(W^{F_1}_{\bullet,\bullet};Z)= \left(\frac{3}{(-K_X)^3}\int\limits_0^{\infty}(P(u)^2\cdot F_1)\cdot\ord_Z(N(u)|_{F_1})du
+\frac{3}{(-K_X)^3}\int\limits_0^{\infty}\int\limits_0^{\infty}\vol(P(u)|_{F_1}-vZ)dv du\right)\\
=\frac{3}{14}\left(\int\limits_0^{\frac{1}{2}}\int\limits_0^{1}(6-2v-v^2)dv du+\int\limits_0^{\frac{1}{2}}\int\limits_1^{2}(v^2-6v+8)dv du+\int\limits_{\frac{1}{2}}^1\int\limits_0^{2-2u}(-8u^2-v^2+4uv-4v+8) dv du+\right.\\
\left.+\int\limits_{\frac{1}{2}}^{1}\int\limits_{2-2u}^{2u}(-4u^2+8uv-8u-8v+12)dv du+\int\limits_{\frac{1}{2}}^{1}\int\limits_{2u}^{2}((2-v)^2+2(2-2u)(2-v)))dv du\right)=1
\end{multline*}.

We have
\begin{multline*}
S(W^{F_1,Z}_{\bullet,\bullet,\bullet};P)=\frac{3}{(-K_X)^3}\int\limits_0^{\infty}\int\limits_0^{\infty}(P(u,v)\cdot Z)^2dv du+\\
+\frac{6}{(-K_X)^3}\int\limits_0^{\infty}\int\limits_0^{\infty}(P(u,v)\cdot Z)\cdot\ord_P(N'_{F_1}(u)|_Z+N(u,v)|_Z)dv du
\end{multline*}

Since $P\not\in E_1\cup E_2\cup E_4$, we see that $\ord_P(N'_{F_1}(u)|_Z)=0$ and $\ord_P(N(u,v)|_Z)=0$.

So, we obtain \begin{multline*}
S(W^{F_1,Z}_{\bullet,\bullet,\bullet};P)
=\frac{3}{14}\left(\int\limits_0^{\frac{1}{2}}\int\limits_0^{1}(v+1)^2dv du+\int\limits_0^{\frac{1}{2}}\int\limits_1^{2}(3-v)^2dv du+\int\limits_{\frac{1}{2}}^{1}\int\limits_0^{2-2u}(2-2u+v)^2dv du+\right.\\
+\left.\int\limits_{\frac{1}{2}}^{1}\int\limits_{2-2u}^{2u}(4-4u)^2dv du+\int\limits_{\frac{1}{2}}^{1}\int\limits_{2u}^{2}(4-2u-v)^2dv du\right)=\frac{39}{56}.
\end{multline*}

So, $\delta_P(X)\geq1$ (see \ref{Van2}).

Assume that $P$ is the intersection point of $E_3,E_4$.
Let $P(u)$ be the positive part of the Zariski decomposition of $-K_X-uF_1$. According to Proposition \ref{ZD2} we have
$P(u)=D+(1-u)F_1+2F_2$ and $N(u)=0$ for $0\leq u\leq\frac{1}{2}$, and $P(u)=(2-2u)D+(1-u)F_1+2F_2$ and $N(u)=(2u-1)D$ for $\frac{1}{2}\leq u\leq 1$. Note that $P(u)|_{F_1} = -K_{F_1}$ for $0\leq u\leq\frac{1}{2}$.

Let $\epsilon\colon Y\rightarrow F_1$ be the blow-up of intersection point of $E_3$ and $E_4$ and $Z$ be the exceptional divisor. Denote by $P(u,v)$ the positive part of the Zariski decomposition of $\epsilon^*P(u)|_{F_1}-vZ$, and by $N(u,v)$ the negative part of the Zariski decomposition of this divisor, where $P(u)$ is given by Proposition \ref{ZD2}. We obtain
\[
P(u,v) = \begin{cases}E_1+E_2+2(E_3+E_4)+(4-v)Z, \quad \quad \text{for}\quad \quad 0\leq u\leq \frac{1}{2},\ \ 0\leq v\leq 1, \\
E_1+E_2+\frac{5-v}{2}(E_3+E_4)+(4-v)Z, \quad \quad \text{for}\quad \quad 0\leq u\leq \frac{1}{2},\ \ 1\leq v\leq 3, \\
(4-v)(E_1+E_2+E_3+E_4+Z), \quad \quad \text{for}\quad \quad 0\leq u\leq \frac{1}{2},\ \ 3\leq v\leq 4, \\
(2-2u)(E_1+E_2)+2(E_3+E_4)+(4-v)Z, \quad \quad \text{for}\quad \quad \frac{1}{2}\leq u\leq 1,\ \ 0\leq v\leq 2-2u, \\
(2-2u)(E_1+E_2)+\frac{6-2u-v}{2}(E_3+E_4)+(4-v)Z, \quad \quad \text{for}\quad \quad \frac{1}{2}\leq u\leq 1,\ \ 2-2u\leq v\leq 2+2u,\\
(4-v)(E_1+E_2+E_3+E_4+Z), \quad \quad \text{for}\quad \quad \frac{1}{2}\leq u\leq 1,\ \ 2+2u\leq v\leq 4.
\end{cases}
\]
and
\[
N(u,v)=\begin{cases}0\quad \quad \text{for}\quad \quad 0\leq u\leq \frac{1}{2},\ \ 0\leq v\leq 1, \\
\frac{v-1}{2}(E_3+E_4), \quad \quad \text{for}\quad \quad 0\leq u\leq \frac{1}{2},\ \ 1< v\leq 3, \\
(v-3)(E_1+E_2)+(v-2)(E_3+E_4), \quad \quad \text{for}\quad \quad 0\leq u\leq \frac{1}{2},\ \ 3< v\leq 4, \\
0\quad \quad \text{for}\quad \quad \frac{1}{2}\leq u\leq 1,\ \ 0\leq v\leq 2-2u, \\
\frac{2u+v-2}{2}(E_3+E_4), \quad \quad \text{for}\quad \quad \frac{1}{2}\leq u\leq 1,\ \ 2-2u\leq v\leq 2+2u\\
(v-2u-2)(E_1+E_2)+(v-2)(E_3+E_4), \quad \quad \text{for}\quad \quad \frac{1}{2}\leq u\leq 1,\ \ 2+2u< v\leq 4.
\end{cases}
\]

Then \begin{multline*}
S(W^{F_1}_{\bullet,\bullet};Z)= \left(\frac{3}{(-K_X)^3}\int\limits_0^{\infty}(P(u)^2\cdot F_1)\cdot\ord_Z(N(u)|_{F_1})du
+\frac{3}{(-K_X)^3}\int\limits_0^{\infty}\int\limits_0^{\infty}\vol(P(u)|_{F_1}-vZ)dv du\right)\\
=\frac{3}{14}\left(\int\limits_0^{\frac{1}{2}}\int\limits_0^{1}(6-v^2)dv du+\int\limits_0^{\frac{1}{2}}\int\limits_1^{3}(7-2v)dv du+\int\limits_0^{\frac{1}{2}}\int\limits_3^{4}(4-v)^2 dv du+\right.\\
+\int\limits_{\frac{1}{2}}^{1}\int\limits_0^{2-2u}(-2(2-2u)^2-16-(4-v)^2+8(2-2u)+8(4-v))dv du+\\
\left.+\int\limits_{\frac{1}{2}}^{1}\int\limits_{2-2u}^{2+2u}(-(2-2u)^2+2(2-2u)(4-v))dv du+\int\limits_{\frac{1}{2}}^{1}\int\limits_{2+2u}^{4}(4-v)^2dv du\right)=2
\end{multline*}

Let $P'\in Z$ be a point.
We have
\begin{multline*}
S(W^{F_1,Z}_{\bullet,\bullet,\bullet};P')=\frac{3}{(-K_X)^3}\int\limits_0^{\infty}\int\limits_0^{\infty}(P(u,v)\cdot Z)^2dv du+\\
+\frac{6}{(-K_X)^3}\int\limits_0^{\infty}\int\limits_0^{\infty}(P(u,v)\cdot Z)\cdot\ord_{P'}(N'_{F_1}(u)|_Z+N(u,v)|_Z)dv du
\end{multline*}

Note that $\ord_P(N'_{F_1}(u)|_Z)=0$.

We obtain \begin{multline*}
\int\limits_0^{\infty}\int\limits_0^{\infty}(P(u,v)\cdot Z)^2dv du
=\int\limits_0^{\frac{1}{2}}\int\limits_0^{1}((E_1+E_2+2(E_3+E_4)+(4-v)Z)\cdot Z)^2dv du+\\
+\int\limits_0^{\frac{1}{2}}\int\limits_{1}^{3}((E_1+E_2+\frac{5-v}{2}(E_3+E_4)+(4-v)Z)\cdot Z)^2dv du+\int\limits_0^{\frac{1}{2}}\int\limits_{3}^{4}((4-v)(E_1+E_2+E_3+E_4+Z)\cdot Z)^2dv du+\\
+\int\limits_{\frac{1}{2}}^{1}\int\limits_0^{2-2u}(((2-2u)(E_1+E_2)+2(E_3+E_4)+(4-v)Z)\cdot Z)^2dv du+\\
+\int\limits_{\frac{1}{2}}^{1}\int\limits_{2-2u}^{2+2u}(((2-2u)(E_1+E_2)+\frac{6-2u-v}{2}(E_3+E_4)+(4-v)Z)\cdot Z)^2dv du+\\
+\int\limits_{\frac{1}{2}}^{1}\int\limits_{2+2u}^{4}(((4-v)(E_1+E_2+E_3+E_4+Z))\cdot Z)^2dv du=\int\limits_0^{\frac{1}{2}}\int\limits_0^{1}v^2dv du+\int\limits_0^{\frac{1}{2}}\int\limits_{1}^{3}dv du+\int\limits_0^{\frac{1}{2}}\int\limits_{3}^{4}(4-v)^2dv du+\\
+\int\limits_{\frac{1}{2}}^{1}\int\limits_0^{2-2u}v^2dv du+\int\limits_{\frac{1}{2}}^{1}\int\limits_{2-2u}^{2+2u}(2-2u)^2dv du+\int\limits_{\frac{1}{2}}^{1}\int\limits_{2+2u}^{4}(4-v)^2dv du=\frac{1}{6}+1+\frac{1}{6}+\frac{1}{24}+\frac{5}{12}+\frac{1}{24}=\frac{11}{6}
\end{multline*}

Assume that $P'$ is not an intersection point of $Z$ and $E_3$ or $E_4$. Then $\ord_{P'}(N(u,v)|_Z)=0$. So, $S(W^{F_1,Z}_{\bullet,\bullet,\bullet};P')=\frac{11}{28}$. Assume that $P'$ is an intersection point of $Z$ and $E_3$. Then

\begin{multline*}
\int\limits_0^{\infty}\int\limits_0^{\infty}(P(u,v)\cdot Z)\cdot\ord_{P'}(N'_{F_1}(u)|_Z+N(u,v)|_Z)dv du=\int\limits_0^{\frac{1}{2}}\int\limits_{1}^{3}\frac{v-1}{2}dv du+\int\limits_0^{\frac{1}{2}}\int\limits_{3}^{4}(4-v)(v-2)dv du+\\
+\int\limits_{\frac{1}{2}}^{1}\int\limits_{2-2u}^{2+2u}(2-2u)\frac{2u+v-2}{2}dv du+\int\limits_{\frac{1}{2}}^{1}\int\limits_{2+2u}^{4}(4-v)(v-2)dv du=\frac{17}{12}.
\end{multline*}
So, $S(W^{F_1,Z}_{\bullet,\bullet,\bullet};P')=\frac{3}{14}\frac{11}{6}+\frac{6}{14}\frac{17}{12}=1$.

So, $\delta_P(X)\geq1$ (see \ref{Van3}).

\end{proof}

\section{Multiple fiber}

Assume that $\pi\colon X\rightarrow\pp^1\times\pp^1$ has a multiple fiber $C'=2C$. Let $F_1$ and $F_2$ be the fibers of $\pi_1$ and $\pi_2$ such that $F_1$ and $F_2$ contain $C$. Let $P\in C$ and $P\not\in D$.

\begin{lemma}
\label{Mult1} Assume that $P\in C$, $P\not\in D$ and $C$ contains singular points of type $A_1$. Then $\delta_P(X)>1$.
\end{lemma}
\begin{proof}
By Lemma \ref{lem-singular-cubic}, $C$ contains two singular points $Q_1$, $Q_2$ of type $A_1$.
By Lemma \ref{lem-singular-cubic}, there exists a unique $(-1)$-curve $E$ intersecting $C$ outside of singular points.
Remark \ref{lem-duVal_contractions} implies that $E=D\cap F_2$ (note that $E$ cannot pass through singular points of $F_2$ since $E$ is smooth).
According to Proposition \ref{ZD} we have $P(u)|_{F_2}=-K_{F_2}$, $N(u)|_{F_2}=0$ for $0\leq u\leq 1$ and $P(u)|_{F_2}=-K_{F_2}-(u-1)E$, $N(u)=(u-1)E$ for $1<u\leq 2$.
Computing the Zariski decomposition of the divisor $P(u)|_{F_2}-vC=P(u,v)+N(u,v)$, where \[
P(u,v) = \begin{cases}-K_{F_2}-vC, \quad \quad \text{for}\quad \quad 0\leq u\leq 1,\ \ 0\leq v\leq 1, \\
-K_{F_2}-vC-(v-1)E, \quad \quad \text{for}\quad \quad 0\leq u\leq 1,\ \ 1< v\leq 2, \\
-K_{F_2}-vC-(u-1)E, \quad \quad \text{for}\quad \quad 1\leq u\leq 2,\ \ 0\leq v\leq u, \\
-K_{F_2}-vC-(v-1)E, \quad \quad \text{for}\quad \quad 1\leq u\leq 2,\ \ u< v\leq 2.
\end{cases}
\]
and
\[
N(u,v)=\begin{cases}0\quad \quad \text{for}\quad \quad 0\leq u\leq 1,\ \ 0\leq v\leq 1, \\
(v-1)E, \quad \quad \text{for}\quad \quad 0\leq u\leq 1,\ \ 1< v\leq 2, \\
0\quad \quad \text{for}\quad \quad 1\leq u\leq 2,\ \ 0\leq v\leq u, \\
(v-u)E, \quad \quad \text{for}\quad \quad 1\leq u\leq 2,\ \ u< v\leq 2.
\end{cases}
\]

We obtain
\begin{multline*}
S(W^{F_2}_{\bullet,\bullet};C)=\frac{3}{(-K_X)^3}\int\limits_0^{\infty}(P(u)^2\cdot F_2)\cdot\ord_C(N(u)|_{F_2})du
+\frac{3}{(-K_X)^3}\int\limits_0^{\infty}\int\limits_0^{\infty}\vol(P(u)|_{F_2}-vC)dv du=\\
=\frac{3}{14}\int\limits_0^{1}\int\limits_0^{\infty}\vol(-K_{F_2}-vC)dv du+\frac{3}{14}\int\limits_1^{2}\int\limits_0^{\infty}\vol(-K_{F_2}-vC-(u-1)E)dvdu=\\
=\frac{3}{14}\int\limits_0^{1}\int\limits_0^{1}(-K_{F_2}-vC)^2dv du+\frac{3}{14}\int\limits_0^{1}\int\limits_{1}^{2}(-K_{F_2}-vC-(v-1)E)^2dv du+\\
+\frac{3}{14}\int\limits_1^{2}\int\limits_0^{u}(-K_{F_2}-vC-(u-1)E)^2dvdu+\frac{3}{14}\int\limits_1^{2}\int\limits_{u}^{2}(-K_{F_2}-vC-(v-1)E)^2dvdu=\\
=\frac{3}{14}\int\limits_0^{1}\int\limits_0^{1}(3-2v)dvdu+\frac{3}{14}\int\limits_0^{1}\int\limits_{1}^{2}(2-v)^2dvdu+\frac{3}{14}\int\limits_1^{2}\int\limits_0^{u}(3-2v-2(u-1)+2v(u-1)-(u-1)^2)dvdu+\\
+\frac{3}{14}\int\limits_1^{2}\int\limits_{u}^{2}(2-v)^2dvdu=\frac{3}{14}\left(2+\frac{1}{3}+\frac{4}{3}+\frac{1}{12}\right)=\frac{45}{56}<1.
\end{multline*}

Also, we have
\begin{multline*}
S(W^{F_2,C}_{\bullet,\bullet,\bullet};P)=\frac{3}{(-K_X)^3}\int\limits_0^{\infty}\int\limits_0^{\infty}(P(u,v)\cdot C)^2dv du+\\
+\frac{6}{(-K_X)^3}\int\limits_0^{\infty}\int\limits_0^{\infty}(P(u,v)\cdot C)\cdot\ord_P(N'_{F_2}(u)|_C+N(u,v)|_C)dv du=\\
=\frac{3}{14}\int\limits_0^{1}\int\limits_0^{1}1^2dv du+\frac{3}{14}\int\limits_0^{1}\int\limits_{1}^{2}(2-v)^2dv du+\\
+\frac{3}{14}\int\limits_1^{2}\int\limits_0^{u}(2-u)^2dvdu+\frac{3}{14}\int\limits_1^{2}\int\limits_{u}^{2}(2-v)^2dvdu=\\
=\frac{3}{14}\left(1+\frac{1}{3}+\frac{5}{12}+\frac{1}{12}\right)=\frac{11}{28}<\frac{1}{2}.
\end{multline*}

So, $\delta_P(X)\geq\frac{56}{45}>1$ (see Propositions \ref{Van1} and \ref{Van2}).
\end{proof}

\section{Proof of main results}

\begin{proof}[Proof of Theorem \ref{main-theorem}]
Let $X$ be a Fano threefold with Picard rank $3$ and degree $14$, and $P\in X$. Assume that $P\in D$. Then by Lemma \ref{DeltaD} we have $\delta_P(X)\geq\frac{8}{7}$. So, we may assume that $P\not\in D$.
Let $F_2$ be a fiber of $\pi_2$ that contains $P$ and $C$ be a fiber of $\pi$ that contains $P$. Assume that $C$ is not a multiple fiber. Assume that $F_2$ is smooth at $C$. By Lemma \ref{DeltaSmF2} we have $\delta_P(X)>1$. So, we may assume that $F_2$ has a singular point in $C$. Let $F_1$ be a fiber of $\pi_1$ that contains $P$. Then $F_1$ contains $C$. By Lemma \ref{Lemm1}, $F_1$ is smooth. Then by Lemma \ref{DeltaSmF1} we see that $\delta_P(X)\geq1$.

Suppose that there exists a prime divisor $E$ over $X$
with $P\in C_X(E)$ and $\beta_X(E)=0$. Then $A_X(E)/S_X(E)=1$, so equality is attained in the
estimate of Lemma \ref{DeltaSmF1}. By the equality cases in Propositions \ref{Van1}, \ref{Van2} and
\cite[Remark 1.7.32]{ChAll}, applied to the admissible flag whose first member is $F_1$, this implies
\[
\frac{A_X(E)}{S_X(E)}=\frac{1}{S_X(F_1)}.
\]
On the other hand, $X$ is divisorially stable, hence $\beta_X(F_1)>0$. Since $A_X(F_1)=1$, we get
$S_X(F_1)<1$, a contradiction.

So, we may assume that $C$ is a multiple fiber. By the generality assumption $(\star)$, there are two singular points of type $A_1$ on $C\subset F_2$. By Lemma \ref{Mult1} we see that $\delta_P(X)>1$. This completes the proof.
\end{proof}

\begin{proof}[Proof of Corollary \ref{cor-a1-a2}]
According to \cite[Lemma 1.5]{Zh}, the singularities along a multiple fiber of a conic bundle $\pi|_{F_2}\colon F_2\to \mathbb{P}^1$, can have one of the following types: $2A_1$, $A_3$ or $D_m$ for $m\geq 4$. According to our assumption, the latter two cases are not possible. Thus, $F_2$ has singularities of type $A_1$ along a multiple fiber of $\pi|_{F_2}$, and so $X$ satisfies the generality assumption $(\star)$. Hence by Theorem \ref{main-theorem} the variety $X$ is K-stable.
\end{proof}

\end{document}